\documentclass{article}

\usepackage[utf8]{inputenc}
\usepackage{tikz,pgfplots,tikz-network}
\usetikzlibrary{arrows.meta}
\usepackage{amsmath,amssymb,amsfonts,amsthm,mathtools}
\usepackage{bm,comment,tikz,url}
\usepackage{booktabs}
\usepackage{graphicx,color}
\usepackage[algoruled]{algorithm2e}
\usepackage{cleveref}

\newtheorem{theorem}{Theorem}
\newtheorem{corollary}{Corollary}

\theoremstyle{remark}
\newtheorem{remark}{Remark}%

\numberwithin{equation}{section}

\definecolor{mygray}{gray}{0.7}
\newcommand{\wh}{\widehat}

\newcommand{\DA}{D}

\newcommand{\trace}{\hbox{\rm trace}}
\newcommand{\ones}{\bm 1}
\newcommand{\pig}{\boldsymbol{\pi}}
\newcommand{\bd}{\bm{d}}
\newcommand{\be}{\bm{e}}
\newcommand{\bv}{\bm{v}}
\newcommand{\bw}{\bm{w}}

\newcommand{\by}{\bm{y}}
\newcommand{\bz}{\bm{z}}
\newcommand{\bx}{\bm{x}}

\title{The derivative of Kemeny's constant as a centrality measure in undirected graphs}

\author{Dario A. Bini\thanks{Dipartimento di Matematica Universit\`a di Pisa} \and Beatrice Meini\thanks{Dipartimento di Matematica Universit\`a di Pisa} \and Federico Poloni\thanks{Dipartimento di Informatica Universit\`a di Pisa}}

\begin{document}
\maketitle

\begin{abstract}
 Kemeny's constant quantifies a graph's connectivity by measuring the average time for a random walker to reach any other vertex. We introduce two concepts of the directional derivative of Kemeny's constant with respect to an edge and use them to define centrality measures for edges and non-edges in the graph. Additionally, we present a sensitivity measure of Kemeny's constant. An explicit expression for these quantities involving the inverse of the modified graph Laplacian is provided, which is valid even for cut-edges. These measures are connected to the one introduced in [Altafini et al., SIMAX 2023], and algorithms for their computation are included. The benefits of these measures are discussed, along with applications to road networks and link prediction analysis. For one-path graphs, an explicit expression for these measures is given in terms of the edge weights.
\end{abstract}


\section{Introduction}
The significance of edges or nodes in a connected graph $G$ is generally assessed by computing a function of the adjacency matrix \cite{benzi20,higham:book} or by modeling a random walk on the graph using a Markov chain and analyzing its performance metrics. One such metric is Kemeny's constant $\kappa(G)$, which represents the expected number of steps for the chain to travel in the graph from a starting state $i$ to a randomly selected destination state, chosen according to the chain's stationary distribution. This expected time remains the same regardless of the initial state $i$, hence the name "constant" \cite{KemenySnell}. Algebraically, Kemeny's constant is given by the trace of a pseudoinverse of the matrix $I-P$, where $P$ is the transition probability matrix of the Markov chain \cite{hunter14}.

In the context of graphs, Kemeny's constant is used, for example, to evaluate the efficiency of robotic surveillance strategies \cite{Bullo}, to partition nodes into clusters \cite{Berkhout}, to control road network dynamics \cite{Crisostomi:Google}, and more. The effects on Kemeny's constant of the removal or addition of edges in tree graphs are studied in 
\cite{kirkland2023edge}, 
while a structured condition number for Kemeny's constant is introduced in \cite{bk19}.

In~\cite{abcmp}, the variation of Kemeny's constant after removing an edge is used to measure the importance of that edge. More specifically, given an undirected connected weighted graph $G=(V,E)$, where $V=\{1,\ldots,n\}$ is the set of vertices and $E\subset V\times V$ is the set of edges, the measure $c(e)=\kappa(\wh G)-\kappa(G)$ is introduced, where $e=\{p,q\}\in E$, $\wh G$ is the graph obtained from $G$ by removing the edge $\{p,q\}$ and adding two loops $\{p,p\}$ and $\{q,q\}$. 
The addition of the two loops is introduced to avoid the Braess paradox \cite{abcmp}.
The definition of this measure does not apply if  $e$ is a cut-edge, that is, an edge whose removal disconnects the graph, since, in this case, $\widehat G$ is not connected and $\kappa(\widehat G)$ is infinite. To overcome this issue, the authors of \cite{abcmp} introduced a regularized and filtered version of the measure $c(e)$; this version is easy to compute, but, for cut-edges, it is prone to numerical cancellation issues and requires cut-edge detection. The former difficulty is addressed in \cite{bklm}, where an equivalent formula for $c(e)$ is provided for cut-edges only. However, this workaround requires the previous detection of cut-edges in the graph and a different treatment for normal edges and cut-edges.

In this paper, instead of considering the variation of Kemeny's constant after an edge removal and the addition of suitable loops, we analyze its infinitesimal variation by introducing a suitable directional derivative of Kemeny's constant. 
A definition of the derivative of Kemeny's constant has already been introduced in \cite{Berkhout} for cluster separation in graphs, and in \cite{ahpjln} for estimating the effect of the removal of edges in particular graphs. Here, we introduce two definitions of directional derivative of Kemeny's constant with respect to an edge of the graph: a weighted and an unweighted derivative.

The weighted derivative with respect to the edge $e=\{p,q\}\in E$, denoted by $\mu(p,q)$, is shown to be well defined and always positive. It quantifies how quickly Kemeny's constant increases as the weight of edge $e$ decreases.  Roughly speaking, a high value of $\mu(p,q)$ indicates that a small reduction in the edge's weight leads to a significant rise in Kemeny's constant, making it a useful measure of the edge's importance.
 Unlike for the measure introduced in \cite{abcmp},  $\mu(p,q)$ 
 has several key advantages: It is universally defined, i.e., it is well-defined for all edges, including cut-edges, so there is no need to identify them or apply any special rules as filtering and regularization; it is numerically stable since its computation avoids cancellation making it reliable; it is efficient, i.e., it can be computed with the same arithmetic cost as the measure from \cite{abcmp}.

The unweighted derivative with respect to the edge $e=\{p,q\}$, denoted by $\bar\mu(p,q)$, is shown to be well defined and always positive also for non-edges, i.e., if $e\not\in E$. Moreover, we show that $\mu(p,q)=a_{pq}\bar\mu(p,q)$, where $a_{pq}$ is the weight of edge $\{p,q\}$, and the importance of non-edges is also defined through $\bar\mu(p,q)$. This latter property finds applications in link prediction in network analysis \cite{ZHOU}, where the goal is to predict the creation of new edges in a given graph.

Finally, from the definition of the unweighted derivative, we derive a global measure of the sensitivity of Kemeny's constant to small perturbations of the edge weights. This measure can also be related to the analysis of graph robustness, investigated in \cite{bg25}, and to the graph perturbation needed to enforce a specific centrality measure \cite{cdm}.

We introduce effective algorithms for the computation of $\mu(p,q)$ and $\bar\mu(p,q)$, based on the Cholesky factorization of the graph Laplacian matrix. Moreover, we provide explicit expressions of 
$\mu(p,q)$ for the class of one-path graphs in terms of the weights of the edges.

Numerical experiments conducted on the Pisa roadmap demonstrate that the measure $\mu(p,q)$ effectively identifies roads with high centrality values in practice.
The measure $\bar\mu(p,q)$ is applied to the link prediction problem  in star, one-path, and cyclic  graphs, and in a binary tree, to identify new links that would increase the graph's circulability; moreover, it is applied in a collaboration network to predict new collaborations.
The global measure of the sensitivity of Kemeny's constant is applied to star, one-path, and circulant graphs. The numerical experiments show a behaviour which is expected from the topology of the tested graphs.

The paper is organized as follows. After recalling in Section~\ref{sec:prel} some preliminary results, we introduce in Section~\ref{sec:derivative} the new measures $\mu(p,q)$ and $\bar\mu(p,q)$, prove their positivity, relate them to a suitable rank-1 perturbation of the Laplacian of the graph, and give explicit expressions in terms of the centralities introduced in~\cite{abcmp}. Moreover, in Section~\ref{sec:condition} we comment on the condition number of Kemeny's constant, and in Section~\ref{sec:sens} we introduce a global measure of sensitivity.
In Section~\ref{sec:comp}, we analyze some computational issues of these measures and provide an algorithm for their computation.  In Section~\ref{sec:specific}, an explicit expression for $\mu(p,q)$ is given for one-path graphs. In Section~\ref{sec:tests} we present some numerical experiments that show the better qualities of the new measure with respect to those of \cite{abcmp}, together with some applications to the link prediction problem and to the global measure of the sensitivity of Kemeny's constant.
In Section~\ref{sec:conc}, we draw some conclusions.

\section{Notation and preliminaries}\label{sec:prel}

Here and hereafter, we use the following notation: $\ones=(1,1,\ldots,1)^\top$ and $\be_i$  is the $i$th column of the identity matrix $I$. A real matrix $S$ is said to be stochastic if it has nonnegative entries and $S\ones=\ones$. Any stochastic matrix $S$ has spectral radius 1, 1 is an eigenvalue, and if $S$ is irreducible then the eigenvalue 1 is simple, by the Perron-Frobenius theorem.

Given an undirected connected graph $G=(V,E)$, where $V=\{1,\ldots,n\}$ is the set of vertices and $E\subset V\times V$ is the set of edges, denote $A=(a_{ij})$ the associated (symmetric) adjacency matrix so that $a_{ij}=a_{ji}>0$ is the weight of the edge that connects $i$ and $j$ for $\{i,j\}\in E$, and $a_{ij}=0$ if $\{i,j\}\not \in E$. We denote the edges with unordered pairs $\{i,j\}$.  
The size of $A$ is $n=|V|$, where $|V|$ denotes the cardinality of $V$.
In the sequel, we denote
$\bd=A\ones=(d_i)\in\mathbb{R}^n$ and $D=\hbox{diag}(\bd) \in \mathbb{R}^{n\times n}$.

Since $G$ is connected, then $A$ is irreducible, so that $d_i> 0$ for any $i$, and we may define the irreducible stochastic matrix $P=D^{-1}A$, whose entries are the transition probabilities in a random walk over the graph $G$.
The transition probability matrix $P$ defines the Markov chain that models this random walk. 
Denote by $\lambda_i$, $i=1,\ldots,n$, the eigenvalues of $P$.  Since $P$ is stochastic, then $|\lambda_i|\le 1$. Moreover,  since $D^{\frac12}PD^{-\frac12}=
D^{-\frac12}A D^{-\frac12}$ is symmetric, the eigenvalues of $P$ are real and can be ordered as 
\begin{equation} \label{deflambda}
    -1\le \lambda_1\le\lambda_2\le\cdots\le\lambda_{n-1}< \lambda_n=1,
\end{equation}
where the latter strict inequality follows from the irreducibility of $P$ and the Perron-Frobenius theorem.

The Kemeny constant  $\kappa(G)$ of $G$ represents the average travel time in the network described by the graph \cite{Crisostomi:Google}. In other words, $\kappa(G)$ measures the difficulty of circulating in the graph $G$. An algebraic expression 
is given by 
\begin{equation}\label{eq:trace}
\kappa(G)=\hbox{trace}(I-P+\ones \bv^\top )^{-1}-1,\quad \bm v\in\mathbb R^n,\quad \bv^\top\ones=1,
\end{equation}
and this expression is valid independently of the choice of the vector $\bv$ \cite{kemeny,wang2017kemeny}.

Observe that, choosing $\bv=\frac1{\|\bd\|_1}\bd$, we have
\[
I-P+\frac1{\|\bd\|_1}\ones\bd^\top=D^{-\frac12}\left(I-D^{-\frac12}AD^{-\frac12}+\frac1{\|\bd\|_1}\bd^\frac12\bd^{\frac12 \top}\right)D^\frac12,
\]
where $\bd^\frac12=D^\frac12\ones$.
That is,
the matrix $I-P+\frac1{\|\bd\|_1}\ones \bd^\top$ is similar to the symmetric matrix $H$, where
\begin{equation}\label{eq:H}
\begin{split}
   & H=I-D^{-\frac12}AD^{-\frac12}+\frac1{\|\bd\|_1}\bd^\frac12\bd^{\frac12\top}=D^{-\frac12}SD^{-\frac12},\\
   & S=L+\frac1{\|\bd\|_1}\bd\bd^\top,~~L=D-A.
\end{split}
\end{equation}
Note that $S$ is a rank-1 perturbation of the graph Laplacian $L$. This way, we may write
\begin{equation}\label{eq:HS}
\kappa(G)=\hbox{trace}(H^{-1})-1=\hbox{trace}(D^{\frac12}S^{-1}D^{\frac12})-1.
\end{equation}
Since $(I-P)\ones=0$, from the Brauer theorem \cite{brauer1}, we deduce that the eigenvalues of $I-P+\ones\bv^\top$ coincide with the eigenvalues of $I-P$, i.e., $1-\lambda_i$, except for the eigenvalue $1-1=0$, which is replaced by 1.
Therefore  the eigenvalues of $H^{-1}$ are $\gamma_i=\frac1{1-\lambda_i}$, for $i=1,2,\ldots,n-1$ and $\gamma_n=1$, so that $\kappa(G)$ has the equivalent expression
\[
\kappa(G)=\sum_{i=1}^{n-1}\frac 1{1-\lambda_i}.
\]
It is convenient to write $\kappa(P)$ in place of $\kappa(G)$ if we want to point out the dependence of Kemeny's constant on the matrix $P$.

We recall the definition of the Kemeny-based centrality $c(p,q)$ of the edge $\{p,q\}$ of the graph $G$ given in  \cite{abcmp}.
Let $A_{pq}=A+a_{pq}\bw_{pq}\bw_{pq}^\top$, where $\bw_{pq}=\be_p-\be_q$.
The matrix $A_{pq}$ is the adjacency matrix of  the graph $G_{pq}$ obtained from $G$ by removing the edge $\{p,q\}$ and adding the two loops $\{p,p\}$, $\{q,q\}$ both with weight $a_{pq}$. In \cite{abcmp} it is proved that if $\{p,q\}$ is not a cut edge,  $c(p,q)$ defined as
\begin{equation}\label{eq:Sh}
\begin{split}
c(p,q)&=\kappa(P_{pq})-\kappa(P)=\hbox{trace}(D^\frac12({ S}_{pq}^{-1}-S^{-1})D^\frac12),\\
S_{pq}&=S-a_{pq}\bw_{pq}\bw_{pq}^\top,
\end{split}
\end{equation} 
is always positive. If $\{p,q\}$ is a cut-edge, then the matrix $S_{pq}$ is not invertible, and definition \eqref{eq:Sh} does not apply.
Theorem 4.3 in \cite{abcmp} states that if $\{p,q\}$ is not a cut-edge then
\begin{equation}\label{eq:tmp1}
c(p,q)=
a_{pq}\bw_{pq}^\top S^{-1}\DA { S}_{pq}^{-1}\bw_{pq}.
\end{equation}
In the same paper \cite{abcmp} the regularized centrality measure $c_r(p,q)$ is defined for each $r>0$ as
\begin{equation}\label{eq:creg}
c_r(p,q)=a_{pq}\bw_{pq}^\top (r\DA +S)^{-1}\DA (r\DA +{ S}_{pq})^{-1}\bw_{pq}.
\end{equation}
This formula is also valid for cut-edges since $r > 0$, and for a non-cut-edge $\lim_{r\to 0}c_r(p,q)=c(p,q)$.
For a cut-edge $\{p,q\}$, the same authors introduce the filtered measure defined as
\begin{equation}\label{eq:cF}
c^F(p,q)=\lim_{r\to 0}\left(\frac1r-c_r(p,q)\right),
\end{equation}
and show that the limit exists and is positive. We refer to \cite{bklm} for a detailed analysis of this quantity.

\section{Directional derivative of Kemeny's constant}\label{sec:derivative}
For $p<q$ such that $a_{pq}\ne 0$, i.e., the edge $e=\{ p,q\}$ exists in the graph, consider the new graph ${G}_{pq}(t)$ obtained by changing the weight $a_{pq}$ of the edge $e$ to $(1-t)a_{pq}$ and by adding a loop in vertices $p$ and $q$ with weight $ta_{pq}$, where $0<t\le 1$.
The adjacency matrix $A_{pq}(t)$ of $G_{pq}(t)$  can be obtained by means of a rank-one correction of $A$, that is, 
\begin{equation}\label{eq:At}
A_{pq}(t)=A+ta_{pq}\bw_{pq}\bw_{pq}^\top,\quad \bw_{pq}=\be_p-\be_q.
\end{equation}
Since $\bw_{pq}^\top \ones=0$, then  $A_{pq}(t)\ones=A\ones=\bd$, so that the stochastic matrix  $P_{pq}(t)$ associated with $G_{pq}(t)$ is $P_{pq}(t)=\DA ^{-1}A_{pq}(t)$, and the
symmetric matrices $H_{pq}(t)$, $S_{pq}(t)$, and $L_{pq}(t)$ corresponding to $H$, $S$, and $L$ in \eqref{eq:H}, 
are
\begin{equation}\label{eq:Ht}
\begin{split}
&H_{pq}(t)=
H-ta_{pq}\DA ^{-\frac12}\bw_{pq}\bw_{pq}^\top \DA ^{-\frac12},\\
&S_{pq}(t)=S-ta_{pq}\bw_{pq}\bw_{pq}^\top,\\
&L_{pq}(t)=L-ta_{pq}\bw_{pq}\bw_{pq}^\top,\\
&P_{pq}(t)=P+t a_{pq} D^{-1}\bw_{pq} \bw_{pq}^\top.
\end{split}
\end{equation}
Observe that for $t=1$, the graph $G_{pq}(t)$ is obtained by removing the edge $\{p,q\}$ and by adding a loop in the vertex $p$ and one in the vertex $q$ of weight $a_{pq}$.
This way, for $t=1$, the difference $\kappa(P_{pq}(t))-\kappa(P)$ coincides with the centrality measure $c(p,q)$ of the edge $\{p,q\}$ of formula \eqref{eq:Sh} introduced in \cite{abcmp}.

On the other hand, observe that the ratio
\[
\frac{\kappa(P_{pq}(t))-\kappa(P)}{t}
\]
is nothing else but the incremental ratio of the function $\kappa(P_{pq}(t))$, along the direction $a_{pq}\DA^{-1}\bw_{pq}\bw_{pq}^\top$, at $P_{pq}(0)=P$
with incremental step $t$, and its limit for $t\to 0^+$ is the directional derivative of $\kappa(P)$.

Observe also that the stochastic matrix $P_{pq}(t)=\DA ^{-1}A_{pq}(t)$ can be expressed as a convex combination of the stochastic matrices $P$ and $P_{pq}(1)$, that is, $P_{pq}(t)=(1-t)P+t P_{pq}(1)$. A similar formulation of $P_{pq}(t)$ is given in \cite{Berkhout} to define a directional derivative of Kemeny's constant along a different suitable direction that is used there to split the graph into connected components.

We introduce a new measure of the edge $\{p,q\}$ given by
the directional derivative of $\kappa(G_{pq}(t))$ as a function of $t$, namely
\begin{equation}\label{eq:mu}
\mu(p,q)=\lim_{t\to 0^+}\frac{\kappa({ P}_{pq}(t))-\kappa({ P})}{t}.
\end{equation}
We will show that $\mu(p,q)$ is well defined and positive for any edge $\{p,q\}$, even when $\{p,q\}$ is a cut-edge.
We also introduce an unweighted measure, obtained by
 computing the derivative along the unweighted direction $\DA ^{-1}\bw_{pq}\bw_{pq}^\top$, namely,
\[
\bar \mu(p,q)=\lim_{t\to 0^+} \frac{\kappa(P+t\DA ^{-1}\bw_{pq}\bw_{pq}^\top)-\kappa( P)}{t}.
\]
We will show that this latter measure is positive also in the case of a non-edge, $a_{pq}=0$, and that  
\begin{equation}\label{eq:cb}
\mu(p,q)=a_{pq}\bar\mu(p,q).
\end{equation}

We refer to  $\mu(p,q)$ and $\bar\mu(p,q)$ as the {\em weighted} and {\em unweighted} centrality measure, respectively.

The following theorem provides an explicit expression of $\kappa(P_{pq}(t))$ and of all its derivatives $\kappa^{(i)}(P_{pq}(t))$ with respect to $t$, given in terms of the quantities
\begin{equation}\label{eq:gd}
   \alpha(p,q)=\bw_{pq}^\top S^{-1}\bw_{pq},\quad \beta(p,q)=\bw_{pq}^\top(S^{-1}DS^{-1})\bw_{pq};
\end{equation}
it will be used to find explicit expressions for the new measures $\mu(p,q)$ and $\bar\mu(p,q)$.

\begin{theorem}\label{thm:derkem} 
Assume that the graph $G$ is connected. Then  we have
\[\begin{split}
&\kappa(P_{pq}(t))=\kappa(P)+\frac{ta_{pq}}{1-ta_{pq}\alpha(p,q)}\beta(p,q),
\\
&\kappa^{(i)}(P_{pq}(t))=\frac{i!(a_{pq})^i}{(1-ta_{pq}\alpha(p,q))^{i+1}}\alpha(p,q)^{i-1}\beta(p,q),
\end{split}
\]
where $\alpha(p,q)$ and $\beta(p,q)$ are defined in \eqref{eq:gd} and $\kappa^{(i)}(P_t)$ denotes the $i$-th derivative of $\kappa(P_{pq}(t))$ with respect to $t$.
In particular, we have $\kappa'(P_{pq}(t))|_{t=0}=a_{pq}\beta(p,q)$.
\end{theorem}

\begin{proof}
If $a_{pq}=0$ the result is trivial. Otherwise,
applying the Sherman-Woodbury-Morrison formula to $S_{pq}(t)$ in \eqref{eq:Ht} yields
\[
S_{pq}(t)^{-1}=S^{-1}+\frac{ta_{pq}}{1-ta_{pq}\bw_{pq}^\top S^{-1}\bw_{pq}}S^{-1}\bw_{pq}\bw_{pq}^\top S^{-1},
\]
whence we get
\[\begin{split}
\kappa(P_{pq}(t))&=\hbox{trace}(D^\frac12S^{-1}D^\frac12)\\
&+\frac{ta_{pq}}{1-ta_{pq}\bw_{pq}^\top S^{-1}\bw_{pq}}\hbox{trace}(D^\frac12S^{-1}\bw_{pq}\bw_{pq}^\top S^{-1}D^\frac12)-1\\
&=\kappa(P)+\frac{ta_{pq}}{1-ta_{pq}\bw_{pq}^\top S^{-1}\bw_{pq}}\bw_{pq}^\top(S^{-1}DS^{-1})\bw_{pq},
\end{split}\]
which provides the first part of the theorem. The second part, concerning derivatives of $\kappa(P_{pq}(t))$, follows by taking the derivatives of the rational function $ta_{pq}/(1-ta_{pq}\bw_{pq}^\top S^{-1}\bw_{pq})$.
   \end{proof}

In the case of the unweighted increment, where in the expression \eqref{eq:At} of $A_{pq}(t)$, the increment $ta_{pq}\bw_{pq}\bw_{pq}^\top$ is replaced by $t\bw_{pq}\bw_{pq}^\top$, we get analogous expressions to those of $\kappa(P_{pq}(t))$ and its derivatives where $a_{pq}$ is replaced by 1.

From the above result, we may deduce explicit expressions for $\mu(p,q)$ and $\bar{\mu}(p,q)$ together with their derivatives.

\begin{corollary}\label{cor:mu}
    Under the assumption of Theorem \ref{thm:derkem}, we have
        \[
    \begin{array}{ll}
        \mu(p,q)=a_{pq}\beta(p,q) ,
        &\mu^{(i)}(p,q)=i!a_{pq}^i\alpha(p,q)^{i-1}\beta(p,q),\\[1.5ex]
        \bar{\mu}(p,q)=\beta(p,q) ,
        &\bar{\mu}^{(i)}(p,q)=i!a_{pq}^{i-1}\alpha(p,q)^{i-1}\beta(p,q),
    \end{array}
    \]
    where $\alpha(p,q)$ and $\beta(p,q)$ are defined in \eqref{eq:gd}.
\end{corollary}

In other words, the entries of the matrix $\beta(p,q)$, defined in \eqref{eq:gd}, where $p,q=1,\ldots,n$, provide the values of the centrality measure of the edges and the non-edges (in the case of the unweighted measure $\bar\mu(p,q)$)  of a graph.

Relying on the symmetry of $H$ and $S$, we may deduce lower and upper bounds to the value of $\bar\mu(p,q)$. In this regard, we have the following corollary.
\begin{corollary}
    Under the hypotheses of Theorem \ref{thm:derkem}, we have
   {
   \[\begin{split}
    &2\min_i\zeta_i\le\bar\mu(p,q)\le 2\max_i \zeta_i,\\
       & 2\min_i d_i \, \min_i \left(\nu_i^{-2}\right)\le\bar\mu(p,q)\le 
    2\max_i d_i \, \max_i\left(\nu_i^{-2}\right),\\
&    2\min_i \left(d_i^{-1}\right) \, \min\left\{\frac1{(1-\lambda_1)^2},1\right\}\le\bar\mu(p,q)\le 
    2\max_i \left(d_i^{-1}\right) \, \max\left\{\frac1{(1-\lambda_{n-1})^2},1\right\}\\
    \end{split}
    \]}where $\zeta_i$, $\lambda_i$ and $\nu_i$ are the eigenvalues of $S^{-1}DS^{-1}$, $P$ and $S=D-A+\frac1{\|\bd\|_1}\bd\bd^\top$, respectively.
\end{corollary}
\begin{proof}
~From \eqref{eq:gd} and Corollary \ref{cor:mu}, since $\bw_{pq}^\top\bw_{pq}=2$, we get
\[
\bar\mu(p,q)=\bw_{pq}^\top(S^{-1}DS^{-1})\bw_{pq}=2\frac{\bw_{pq}^\top(S^{-1}DS^{-1})\bw_{pq}}{\bw_{pq}^\top\bw_{pq}}.
\]
Since $\bw_{pq}^\top\bw_{pq}=2$, this implies the first pair of inequalities. 
 Moreover,
setting $\by=S^{-1}\bw_{pq}$ we get
    \[
    \bar\mu(p,q)=2\frac{\by^\top D\by}{\by^\top \by}\cdot \frac{\bw_{pq}^\top S^{-2}\bw_{pq}}{\bw_{pq}^\top \bw_{pq}}.
    \]
    The first Rayleigh quotient is in the interval $[\min_{i}d_i,\max _{i}d_i]$, the second one is in the interval $[\min_i\left(\nu_i^{-2}\right),\max_i\left(\nu_i^{-2}\right)]$, where $\nu_i$ are the eigenvalues of $S$. From this, we obtain the first two inclusions.
    To prove the last inclusion, we rely on the identity $S=D^\frac12 HD^\frac12$ of \eqref{eq:H}, so that we may write   $\bar\mu(p,q)=\bw_{pq}^\top(D^{-\frac12}H^{-2}D^{-\frac12})\bw_{pq}$. Applying the same argument as before yields
\[
\bar\mu(p,q)=2\frac{\bz^\top H^{-2}\bz}{\bz^\top\bz}\cdot \frac{\bw_{pq}^\top D^{-1}\bw_{pq}}{\bw_{pq}^\top\bw_{pq}},\quad \bz=D^{-\frac12}\bw_{pq}.
\]
The first Rayleigh quotient is in the range $[\min \gamma_i^{2},\max_i \gamma_i^{2}]$, where $\gamma_i$ are the eigenvalues of $H^{-1}$. Recall that $H$ is similar to $D^{-1}S=I-P+\frac1{\|\bd\|_1}\ones\bd^\top$ so that  $\gamma_i=\frac1{1-\lambda_i}$, $i=1,\ldots,n-1$, $\gamma_n=1$. This concludes the proof.
\end{proof}

An interesting observation stems from comparing the expression of $\mu(p,q)$ given in Corollary~\ref{cor:mu} with the centrality measure given in \cite{abcmp}. Indeed, Theorem~4.3 in \cite{abcmp}, rephrased in our notation, states that if $\{p,q\}$ is not a cut-edge, then the centrality measure given in \cite{abcmp} can be rewritten as in
\eqref{eq:tmp1}.
This expression of $c(p,q)$ allows us to relate $c(p,q)$ to $\mu(p,q)$ as shown in the following theorem.
\begin{theorem}\label{th:2}
    Assume that the graph $G$ is connected, $a_{pq}\ne 0$, and $\{p,q\}$ is not a cut-edge. Then, we have
    \[ 
       c(p,q)=\frac{\mu(p,q)}{1-a_{pq}\bw_{pq}^\top S^{-1}\bw_{pq}}.
    \]

\begin{proof}
~Rewriting $ S_{pq}=(I-a_{pq}\bw_{pq}\bw_{pq}^\top S^{-1})S$ and applying the Sherman-Woodbury-Morrison formula yields
\[
{ S}_{pq}^{-1}=S^{-1}\left(I+\frac{a_{pq}}{1-a_{pq}\bw_{pq}^\top S^{-1}\bw_{pq}}\bw_{pq}\bw_{pq}^\top S^{-1}\right).
\]
Plugging the latter expression into \eqref{eq:tmp1}, because of Theorem \ref{thm:derkem}, yields the result.
\end{proof}
\end{theorem}

\begin{remark}\label{rem:1}
    As shown in \cite{abcmp}, if $\{p,q\}$ is a cut-edge, then $c(p,q)=\infty$. Therefore, from Theorem \ref{th:2} we find that, for a cut-edge $\{p,q\}$, necessarily one has $a_{pq}\bw_{pq}^\top S^{-1}\bw_{pq}=1$.
\end{remark}

We may also relate the measure $\mu(p,q)$ with  the filtered centrality $c^F(p,q)$  defined in \cite{abcmp}, see \eqref{eq:cF}. 

Consistently with the definition of $c_r(p,q)$, we define the {\em regularized derivative-based centrality} $\mu_r(p,q)$ as
\[
\mu_r(p,q)=a_{pq}\bw_{pq}^\top S_r^{-1}\DA S_r^{-1} \bw_{pq},\quad S_r=S+r\DA,
\]
so that $\mu_0(p,q)=\mu(p,q)$.
A simple formal manipulation relying on the second equation in \eqref{eq:Sh} and the Sherman-Woodbury-Morrison formula shows that
\[
(rD+{ S}_{pq})^{-1}=S_r^{-1}+\frac{a_{pq}}{1-a_{pq}\bw_{pq}^\top S_r^{-1}\bw_{pq}}S_r ^{-1}\bw_{pq}\bw_{pq}^\top S_r^{-1} ,
\]
 so that from \eqref{eq:creg} we obtain
\begin{equation}\label{eq:mur}
\begin{split}c_r(p,q)=&
\mu_r(p,q)\left(1+\frac{a_{pq}}{1-a_{pq}\bw_{pq}^\top S_r^{-1}\bw_{pq}}\cdot\bw_{pq}^\top S_r^{-1}\bw_{pq}\right)\\
=&
\frac{1}{1-a_{pq}\bw_{pq}^\top S_r^{-1}\bw_{pq}}\mu_r(p,q).
\end{split}
\end{equation}

The above expression allows us to relate the \emph{filtered centrality measure} $c^F(p,q)$, defined in  \eqref{eq:cF}, 
to our measure $\mu(p,q)$. We have the following
\begin{theorem} If $G$ is connected, $a_{pq}\ne 0$, and $\{p,q\}$ is a cut-edge,  then
    \[
c^F(p,q)=2\frac{\bw_{pq}^\top S^{-1}\DA S^{-1}\DA S^{-1}\bw_{pq}} {\mu(p,q)}.
\]
\end{theorem}
\begin{proof}
~Since $S_r^{-1}=S^{-1}(I+r\DA S^{-1})^{-1}=S^{-1}-rS^{-1}\DA S^{-1}+O(r^2)\doteq S^{-1}-rS^{-1}\DA S^{-1}$, where $\doteq$ means equality up to $O(r^2)$ terms, then we may write
\[
a_{pq}\bw_{pq}^\top S_r^{-1}\bw_{pq}\doteq a_{pq}\bw_{pq}^\top S^{-1}\bw_{pq} -r\mu(p,q)=1-r\mu(p,q),
\]
where the last equality follows from Remark \ref{rem:1}.
 Whence, in view of \eqref{eq:mur} we get 
\[\begin{split}
\frac 1r-c_r(p,q)&\doteq \frac1r-\frac1r\frac{\mu_r(p,q)}{\mu(p,q)}.
\end{split}
\]
Now, it remains to expand $\mu_r(p,q)$ as a power series in $r$.
Since $(r\DA +S)^{-1}\doteq S^{-1}-rS^{-1}\DA S^{-1}$,  we have 
\[
\mu_r(p,q) \doteq \mu(p,q)-2r\bw_{pq}^\top S^{-1}\DA S^{-1}\DA S^{-1}\bw_{pq}.
\]
This completes the proof.
\end{proof}

Observe that, setting $\by_{pq}=\DA ^\frac12S^{-1}\bw_{pq}$ we get 
\[
c^F(p,q)=2\frac{\by_{pq}^\top \DA ^\frac12S^{-1}\DA ^\frac12\by_{pq}}{\by_{pq}^\top\by_{pq}} = 2\frac{\by_{pq}^\top H^{-1}\by_{pq}}{a_{pq}\by_{pq}^\top\by_{pq}},
\]
moreover, the Rayleigh quotient $\by_{pq}^\top H\by_{pq}/\by_{pq}^\top\by_{pq}$  is in between the maximum and the minimum eigenvalue of $H^{-1} = \DA ^\frac12S^{-1}\DA ^\frac12$. Recalling \eqref{deflambda}, these eigenvalues are $\gamma_i=\frac1{1-\lambda_i}$, $i=1,\ldots,n-1$, $\gamma_n=1$, and we get the following bounds.
\begin{corollary} If $A$ is irreducible and $\{p,q\}$ is a cut-edge, then 
    \[
\frac{2}{a_{pq}}\min\left\{1,\frac1{1-\lambda_1}\right\}\le c^F(p,q)\le \frac2{a_{pq}}\max\left\{1,\frac1{1-\lambda_{n-1}}\right\}.
\]
\end{corollary}

\subsection{Condition number}\label{sec:condition}
Observe that in the graph with adjacency matrix $A-t\bw_{pq}\bw_{pq}^\top$,  the edge $\{p,q\}$ has a weight $a_{pq}+t$ and the edges $p$ and $q$ have weights $a_{pp}-t$ and $a_{qq}-t$, respectively. The derivative with respect to $t$ of the Kemeny constant of this graph expresses the absolute variation consequent to the absolute perturbation $t$ of $a_{pq}$ and $a_{qp}$ and to the perturbation $-t$ in $a_{pp}$ and in $a_{qq}$, i.e.,  in the loops in the nodes $p$ and $q$.
This way, we may look at this derivative at $t=0$ as the condition number of the Kemeny constant with respect to an absolute perturbation on $a_{pq}$,  when this perturbation is applied also to the loops in $p$ and $q$ with negative weight.
Relying on Theorem \ref{thm:derkem}, it is easy to verify that this derivative at $t=0$ is given by 
$-\beta(p,q)=-\bw_{pq}^\top(S^{-1}DS^{-1})\bw_{pq}$.  This shows that increasing the weight of $a_{pq}$ and $a_{qp}$, and decreasing by the same amount the weights $a_{pp}$ and $a_{qq}$ leads to a decrease of Kemeny's constant. Moreover, the larger $\beta(p,q)$, the stronger is the variation, and consequently the more ill-conditioned is Kemeny's constant.

\subsection{A global measure of sensitivity}\label{sec:sens}
By following the way the total communicability is defined \cite{bk2013,benzi20}, we 
can take 
the arithmetic mean of the measure $\bar\mu(p,q)$ over each pair $(p,q)$
(including non-edges and self-loops, where we assume $\bar\mu(p,p)=0$), namely
\[
\zeta(G)=\frac{\sum_{p,q=1}^n\bar\mu(p,q)}{n^2},
\]
as a global measure of the sensitivity of Kemeny's constant of a graph $G$ to perturbations of its edges.
A direct computation based on Theorem~\ref{thm:derkem} shows that
\begin{equation}\label{eq:global}
\zeta(G)=\frac1{n^2}\left(n\;\trace(S^{-1} D S^{-1})-\ones^\top
    S^{-1} D S^{-1} \ones\right).
\end{equation}
Observe that the above expression involves the trace and the sum of the entries of the matrix $S^{-1}DS^{-1}$; the trace resembles the Estrada index~\cite{benzi20}, while the sum of the entries resembles total communicability~\cite{bk2013}.

\section{Computational issues}\label{sec:comp}
This section discusses the issues encountered in the actual computation of the quantity $\mu(p,q)$.

From Corollary~\ref{cor:mu} and equation \eqref{eq:gd}, we deduce that
\begin{equation}\label{eq:formula}
\mu(p,q)= a_{pq}\bw_{pq}^\top S^{-1}\DA S^{-1}\bw_{pq} =a_{pq}\sum_{r=1}^n d_r(\sigma_{r,p}-\sigma_{r,q})^2,
\end{equation}
where $\bw_{pq}=\be_p-\be_q$, the matrix $S$ is defined in \eqref{eq:HS}, and $\sigma_{ij}$ are the entries of $S^{-1}$. Moreover, 
\begin{equation}\label{eq:formula1}
\bar\mu(p,q)=\bw_{pq}^\top S^{-1}\DA S^{-1}\bw_{pq} =\sum_{r=1}^n d_r(\sigma_{r,p}-\sigma_{r,q})^2.
\end{equation}

The major effort in computing $\mu(p,q)$ through \eqref{eq:formula} consists in computing the vector $\bm x=S^{-1}\bw_{pq}$, that is, 
solving the system $S\bm x = \bw_{pq}$. If the goal is to compute the centrality measure of a single
edge $\{ p, q\}$, then one can compute the Cholesky factorization $S=R^\top R$ of $S$, and then solve the two triangular systems with matrix $R^\top$ and $R$, respectively.  The cost of this approach is dominated by the computation of the Cholesky factorization of $S$, that is, $\frac13 n^3+O(n^2)$ arithmetic operations
(ops) for a general matrix $A$ \cite[Appendix C]{higham:book}.
 Alternatively, one can apply any (preconditioned) iterative method to solve the system $S\bm x=\bw_{pq}$, taking advantage of the possibly sparse structure of $A$.
 However, the cost of the former approach based on Cholesky factorization can be reduced by exploiting the possible sparsity of the matrix $A$ as follows.

 First, a symmetric sparse matrix can be reduced to a banded form by a permutation of rows and columns \cite{band}. Therefore, without loss of generality, we may assume that $A$ is banded with bandwidth $2b-1$, i.e., $a_{ij}=0$ if $|i-j|\ge b$. Secondly, we may exploit that $L=\DA -A$ is a weakly diagonally dominant and irreducible M-matrix that is also singular. This fact implies that its leading principal submatrix $L_{n-1}$ of $L$, with size $n-1$, is positive definite; therefore, it admits the Cholesky factorization
 $L_{n-1}=R_{n-1}^\top R_{n-1}$ where $R_{n-1}$ is an $(n-1)\times(n-1)$ upper triangular matrix.

Thus, we may write the Cholesky factorization of $L$ as
\[
L=R^\top R,\quad R=
\left[\begin{array}{c|c}
    R_{n-1}&\begin{array}{c}v_1\\v_2\\ \vdots\\ v_{n-1}\end{array}\\ \hline
    0 ~ 0\cdots0&0
\end{array}\right].
\]
    This expression allows us to compute the entries $v_1,v_2,\ldots,v_{n-1}$ of the last column the Cholesky factor $R$ of $L$ by solving the linear system
\begin{equation}\label{eq:Rv}
R^\top_{n-1}\begin{bmatrix}
    v_1\\ \vdots\\ v_{n-1}
\end{bmatrix}=\begin{bmatrix}
    \ell_{1n}\\ \vdots\\ \ell_{n-1,n}
\end{bmatrix}
\end{equation}
where $(\ell_{ij})_{ij}$ denote the entries of $L$.

Now, the system $S\bx=\bw_{pq}$ can be equivalently rewritten as 
\begin{equation}\label{eq:xw}
\left(L+\frac1{\|\bd\|_1}\bd\bd^\top\right)\bx=\bw_{pq},
\end{equation}
that is,
\[
R^\top R\bx=\bw_{pq}-\xi\bd,\quad \xi=(\bd^\top \bx)/\|\bd\|_1.
\]
On the other hand, since
$\ones^\top L=0$ and $\ones^\top \bw_{pq}=0$, by multiplying on the left \eqref{eq:xw} by $\ones^\top$, 
 we get $\xi=\bd^\top\bx=0$, so that the system \eqref{eq:xw} turns into
\begin{equation}\label{eq:sys}
R^\top R\bx=\bw_{pq},~~
\bd^\top\bx=0.
\end{equation}
Let $\by\in\mathbb{R}^{n-1}$ be the vector solving the nonsingular system 
\begin{equation}\label{eq:Ry}
R_{n-1}^\top\by={\overline{\bw}_{pq}}, \quad {\overline{\bw}_{pq}}=(\bw_{pq})_{i=1,\ldots,n-1} 
\end{equation}
From the condition $R^\top R\bx=\bw_{pq}$ we get
$
\left[\begin{array}{c|c}
R_{n-1}&\bv \\ 
\end{array}\right]\bx=\by
$,
so that, adding the condition $\bd^\top\bx=0$, we may rewrite the system \eqref{eq:sys} as
\[
\left[\begin{array}{c|c}
R_{n-1}&\bv\\ \hline \bd_{n-1}^\top&d_n
\end{array}\right]
\left[\begin{array}{c} \bx_{n-1}\\ \hline x_n\end{array}\right]=\left[\begin{array}{c}
    \by\\
     \hline 0
\end{array}\right],\quad \bx_{n-1}=\begin{bmatrix}x_1\\ \vdots\\ x_{n-1}\end{bmatrix},
\]
where $\bd_{n-1}=(d_1,\ldots,d_{n-1})^\top$.
 Setting $\bz$ the solution of the system
\begin{equation}\label{eq:Rz}
\bz^\top R_{n-1}=\bd_{n-1}^\top,
\end{equation}
we arrive at
\begin{equation}\label{eq:Rx}
\left[\begin{array}{c|c}
R_{n-1}&\bv\\ \hline 0&d_n-\bz^\top\bv
\end{array}\right]
\left[\begin{array}{c} \bx_{n-1}\\ \hline x_n\end{array}\right]
=\left[\begin{array}{c}
    \by\\
     \hline -\bz^\top\by
\end{array}\right].
\end{equation}
We may conclude with the following expression for the solution $\bx^\top=[\bx_{n-1}^\top\  x_n]$:
\[
x_n=(\bz^\top\by)/(\bz^\top\bv - d_n),~~
    \bx_{n-1}=R_{n-1}^{-1}(\by-x_n\bv).
\]
The overall computation is summarized in Algorithm \ref{algo:1}.

\begin{remark}\label{rem:Rv}
Since $L=(\ell_{ij})_{ij}$ is banded, then $\ell_{in}=0$ for $i=1,\ldots,n-b$, and since $R_{n-1}^\top$ is lower triangular, then $v_i=0$ for $i=1,\ldots,n-b$ and the vector $(v_{n-b+1},\ldots,v_{n-1})^\top$ solves a $(b-1)\times (b-1)$ triangular system.
\end{remark}

\begin{remark}\label{rem:Ry}
    Since $\bw_{pq}=\be_p-\be_q$, $p<q$, then the first $p-1$ components of $\bw_{pq}$ are zero,
    and since $R_{n-1}^\top$ is lower triangular, then $w_i=0$ for $i=1,\ldots,p-1$ and the linear system \eqref{eq:Ry} reduces to an $(n-p)\times(n-p)$ triangular system.
\end{remark}

{
\begin{algorithm}[ht]\small
\textbf{Input:} The $n\times n$ symmetric irreducible adjacency matrix $A$, the integers $p,q\in\{1,\ldots,n\}$, $p<q$.

{\bf Output:} The vector $\bm x=S^{-1}\bw_{pq}$.

\medskip

{\bf Preprocessing} (independent of $p,q$):
\begin{enumerate}
        \item Compute $\bd=A\ones$ and the graph Laplacian matrix $L=D-A$; 
        \item set $L_{n-1}$ the leading principal submatrix of size $n-1$ of $L$, and 
        compute\\
         the Cholesky factor $R_{n-1}$ of $L_{n-1}$, i.e., such that
        $L_{n-1}=R_{n-1}^\top R_{n-1}$, where $R_{n-1}$ is upper triangular; 
        \item Compute the vectors $\bv$ and $\bz$ by solving the triangular systems
        \eqref{eq:Rv}, \eqref{eq:Rz}\\
         in view of Remark \ref{rem:Rv};
        \item Compute $\rho=\bz^\top\bv - d_n$; 
    \end{enumerate}

{\bf Processing} (dependent on $p,q$):
    \begin{enumerate}
        \item Compute $\by$ by solving \eqref{eq:Ry} relying on Remark \ref{rem:Ry};
        \item Compute $x_n=\bz^\top \by/\rho$; 
        \item Solve the upper triangular system $R_{n-1}\bx_{n-1}=\by-x_n\bv$.
    \end{enumerate}
    
\caption{\small Solve the system $S\bm x=\bm w_{pq}$, where $S=\DA -A+\frac1{\|\bd\|_1}\bd\bd^\top$, $\DA =\hbox{diag}(\bd)$, $\bd=A\ones$, and $\bm w_{pq}=\be_p-\be_q$, $p<q$.}\label{algo:1} 
\end{algorithm}
}

We can perform a complexity analysis of Algorithm~\ref{algo:1} by evaluating the number of arithmetic operations required for the computation. In this analysis, we assume that the matrix $A$ has bandwidth $2b-1$.

Concerning the preprocessing stage, computing $\bd$ costs $(b-1)(2n-(b-1))$ arithmetic operations (ops), while at most $n$ subtractions are needed to compute $L$.
The Cholesky factorization of a positive semidefinite matrix, with bandwidth $2b-1$, costs  $\frac13nb^2$. We refer to the book \cite[Chapter 10]{higham:book1} for the nice computational and numerical features of the Cholesky factorization. We also point out that in our case, the matrix $L_{n-1}$ is a nonsingular M-matrix. This feature adds nice numerical properties to the algorithm for Cholesky factorization if complemented with the GTH 
trick \cite{gth} to avoid possible numerical cancellation.

The vector $\bv$ in the preprocessing stage can be computed by solving a triangular system of size $b$ (see Remark \ref{rem:Rv}), for the cost of $b^2$ ops, while the computation of $\bz$ requires the solution of the triangular system \eqref{eq:Rz} of size $n-1$ with bandwidth $b$ for the cost of $n(2b-1)-b^2+b$ ops. Finally, computing $\rho=\bz^\top\bv - d_n$ costs just $2b-1$ ops since $\bv$ has only the last $b-1$ components nonzero.
Thus, the overall cost of the preprocessing stage is $n(\frac13 b^2+4b-3)-b^2+5b-2$ ops and is dominated by the cost of Cholesky factorization.

Concerning the processing stage, because of Remark \ref{rem:Ry}, step~1 requires the solution of an $(n-p)\times(n-p)$ triangular banded system that costs roughly $2b(n-p)$ ops; 
 step~2 requires computing the scalar product $\bz^\top\by$ which costs $2(n-p)$ ops in view of Remark \ref{rem:Ry}. Step~3 requires solving an upper triangular system of size $n$ with bandwidth $b$; this costs roughly $2bn$ ops.
That is, the overall cost of the processing stage, excluding preprocessing, is $2(n-p)(b+1)$ operations.

To compute the centrality of the edge $\{p,q\}$,  it is enough to compute $a_{pq}\bx^\top D\bx$ that requires $3n$ additional ops. Therefore, the overall cost of computing $\mu(p,q)$, once the preprocessing stage has been performed, is  $n(4b+5)-2p(b+1)$ ops. 

To compute the centralities of all the edges, one can repeat the processing part for all the pairs $\{p,q\}$ corresponding to an edge. An upper bound to this computational cost is $n(4b+3 )|E|+\frac13 nb^2$, where $|E|<nb$ is the number of edges in the graph.

In Section \ref{sec:specific}, we will provide explicit expressions of $\mu(p,q)$ for one-path graphs as performed in \cite{bklm} for $c(e)$, and show the good behaviour of the new measure.

\section{One-path graphs}\label{sec:specific}
Consider the case of a one-path graph formed by $n$ vertices connected as in Figure \ref{fig:1}.

\begin{figure}
    \centering
\begin{tikzpicture}[
      mycircle/.style={
         circle,
         draw=black,
         fill=gray,
         fill opacity = 0.3,
         text opacity=1,
         inner sep=0pt,
         minimum size=5pt,
         font=\small},
      myarrow/.style={},  
     node distance=0.7cm and 1.3cm
      ]
      \node[mycircle] (c1) [label=above:$1$]{};
      \node[mycircle,right=of c1] (c2) [label=above:$2$]{};
      \node[mycircle,right=of c2] (c3) [label=above:$3$]{};
      \node[mycircle,right=of c3] (c4) [label=above:$4$]{};
      \node[mycircle,right=of c4] (c5) [label=above:$5$]{};
      \node[mycircle,right=of c5] (c6) [label=above:$6$]{};
   \draw[myarrow] (c1) -- (c2);
   \draw[myarrow] (c4) -- (c5);
   \draw[myarrow] (c2) -- (c3);
   \draw[myarrow] (c3) -- (c4);
  \draw[myarrow] (c5) -- (c6);
    \end{tikzpicture}\caption{\small One-path graph formed by $n=6$ vertices.}\label{fig:1}
\end{figure}
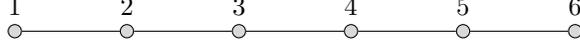

The adjacency matrix $A=(a_{ij})$ is tridiagonal with $a_{i,i+1}=a_{i+1,i}\ne 0$, $a_{ij}=0$, otherwise.
Let $P=\DA ^{-1}A$, where $\DA =\hbox{diag}(\bm d)$, $\bm d=A\ones$, then
\begin{equation}\label{eq:Pqbd}
    P=\begin{bmatrix}
 \theta_0& \lambda_0&\\
\nu_1 & \theta_1 & \lambda_1\\
&\ddots & \ddots & \ddots \\
&&\nu_{n-2} & \theta_{n-2} &\lambda_{n-2}\\
&&&\nu_{n-1}&\theta_{n-1}
\end{bmatrix},
\end{equation}
where $\theta_0=1-\lambda_0\ge 0$, $\theta_i=1-\lambda_i-\nu_i\ge 0$ for $i=1,\ldots,n-2$, $\theta_{n-1}=1-\lambda_{n-1}$, and $0<\lambda_i<1$, $0<\nu_i<1$.
For this graph, we have $\theta_i\ne 0$ if and only if there is a loop $\{i,i\}$ at node $i$.
From \cite[Theorem 5.1]{bini2018kemeny}, we find that Kemeny's constant of $P$ is given by
\begin{equation}\label{eq:kqbd}
\kappa(P)=\sum_{k=0}^{n-2}\frac1{\lambda_k \pi_k}\sigma_k(1-\sigma_k),\quad \sigma_k= 
\sum_{j=0}^{k}\pi_j,
\end{equation}
where $\pig=(\pi_i)_{i=0,\ldots,n-1}$ is the steady state vector of $P$, i.e., the stochastic solution of $\pi^\top P=\pi^\top$.
Recall also that
\begin{equation}\label{eq:pi}
\pi_0=\left(1+\sum_{i=1}^{n-1}\prod_{\ell=0}^{i-1}\frac{\lambda_\ell}{\nu_{\ell+1}}\right)^{-1},~~
\pi_i=\pi_0\prod_{\ell=0}^{i-1}\frac{\lambda_\ell}{\nu_{\ell+1}},\quad i>0.
\end{equation}
Now consider an integer $q$, $1\le q\le n-1$, and define
\[
A_q(t)=A+ta_{q,q+1}\bw_{q,q+1}\bw_{q,q+1}^\top.
\]
Observe that $A_q(t)\ones =\bm d$ so that denoting $P_q(t)=\hbox{diag}(\bm d)^{-1}A_q(t)$ we have
\begin{equation}\label{eq:Pqbdt}
    P_q(t)=\begin{bmatrix}
 \theta_0(t)& \lambda_0(t)&\\
\nu_1(t) & \theta_1(t) & \lambda_1(t)\\
&\ddots & \ddots & \ddots \\
&&\nu_{n-2}(t) & \theta_{n-2}(t) &\lambda_{n-2}(t)\\
&&&\nu_{n-1}(t)&\theta_{n-1}(t)
\end{bmatrix},
\end{equation}
where
$\lambda_i(t)=\lambda_i$, $\nu_{i+1}(t)=\nu_{i+1}$ are independent of $t$ for $i\ne q-1$, while
$
\lambda_{q-1}(t)=\lambda_{q-1}(1-t),\quad \nu_{q}(t)=\nu_q(1-t)$.
These expressions, together with \eqref{eq:pi} imply that $\pig^\top P_q(t)=\pig^\top$, consequently
\[
\kappa(P_q(t))=\sum_{k=0}^{n-2}\frac1{\lambda_k(t)\pi_k}\sigma_k(1-\sigma_k).
\]
The $k$th term in the above summation is independent of $t$ if $k\ne q-1$ while the $(q-1)$st term is
\[
\frac{1}{(1-t)\lambda_{q-1}\pi_{q-1}}\sigma_{q-1}(1-\sigma_{q-1})
\]
so that
\[
\kappa(P_q(t))'|_{t=0}= 
\frac{\sigma_{q-1}(1-\sigma_{q-1})}{\lambda_{q-1}\pi_{q-1}}.
\]
We conclude with the following theorem.
\begin{theorem}
    For the centrality $\mu(q,q+1)$ of the edge $\{q,q+1\}$ in the one-path graph associated with the matrix $P$ of \eqref{eq:Pqbd}, we have
    \[
\mu(q,q+1)=\frac{\sigma_{q-1}(1-\sigma_{q-1})}{\lambda_{q-1}\pi_{q-1}},
    \]
    where $\sigma_k=\sum_{j=0}^k\pi_j$,  $\pi^\top P=\pi^\top$, and the components $\pi_i$ of $\pi$ satisfy \eqref{eq:pi}.
    Moreover, in the case of unitary weights, we have $\lambda_0=\nu_{n-1}=1$, $\lambda_q=\nu_{n-q-1}=1/2$ for $q>0$, so that
\[
\pig^\top=\left[1/2,1,\ldots,1,1/2\right]/(n-1),\quad \sigma_q=\frac12(2q+1)/(n-1).
\]
These values give
\begin{equation}\label{eq:onepath}
\mu(q,q+1)=\kappa(P_q(t))'|_{t=0}= \frac1{2(n-1)}(2q-1)(2n-2q-1).
\end{equation}
\end{theorem}

It is interesting to observe that the values given in \eqref{eq:onepath} are located on the graph of a parabola as the values of the measure given in \cite[Theorem 7]{bklm}, which essentially differ for a normalizing factor and lower order terms.

\section{Numerical tests}\label{sec:tests}

We have performed some experiments to test our centrality measures. A first test concerns the good physical behavior of the measure for assessing the importance of roads in a given road map. In Figure~\ref{fig:pisa}, we compare our measure to the measure of \cite{abcmp} on the road map of Pisa.  
The two measures have been linearly normalized so that their values are in the range $[0,1]$, i.e., the transformation $x\to (x-a)/(b-a)$ has been applied, where $a$ and $b$ are the minimum and the maximum values taken by the measure. Regarding the association of the color with the numerical value of the measure, we have considered the square root of the measure to better highlight the intermediate values.

From this comparison, we can see that the two measures behave similarly; in particular, both highlight the most important connections, such as the bridges on the river Arno and the railway bridges. 
Moreover, the relative differences of the two measures after linear normalization are between $7.2\times 10^{-8}$ and $2.8\times 10^{-3}$.  Additionally, the histograms of the two measures in Figure \ref{fig:histo}, which report, on a log scale in the $y$ axis, the number of values falling in each of 50 bins, appear very similar.
However, the new measure assigns fewer edges to the rightmost bins, which corresponds to larger values, providing a more definite ranking of the most important roads. This clarifies that the bridges over the Arno river are the most crucial, while the less connected peripheral railroad bridges are assigned a lower importance. We believe this new measure offers a more realistic and useful importance rating.

Using MATLAB on a 12-core, 3.4 GHz parallel server, our new implementation significantly reduces computation time. For a smaller dataset with n=1404 nodes and m=1794 edges, the computation takes about 0.1 seconds. For the entire Tuscany region ({\small $n=1,150,138$, $m=1,223,248$}), the computation is completed in 62 minutes, a significant improvement over the 18 hours reported in the implementation given in \cite{abcmp}. This substantial speed-up is primarily due to our enhanced code implementation.

\begin{figure}
    \centering
    \includegraphics[width=0.45\linewidth]
    {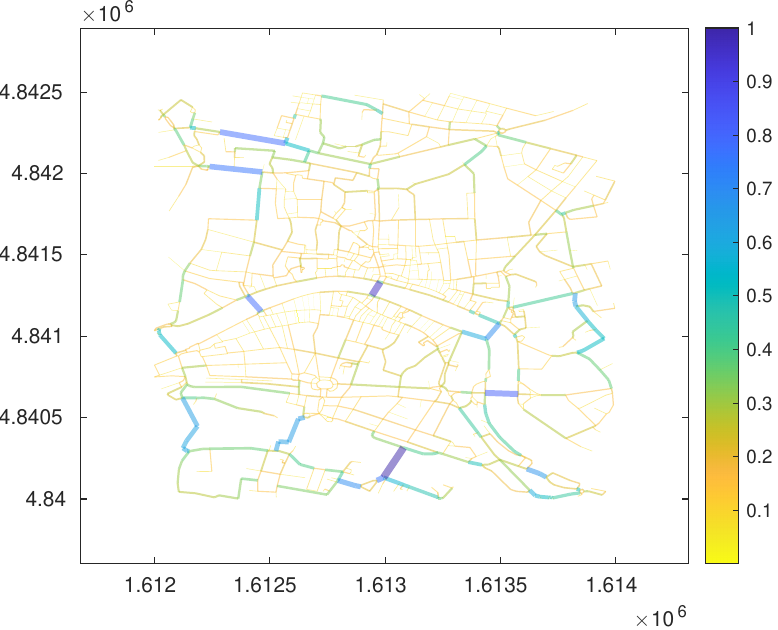}~~
     \includegraphics[width=0.45\linewidth]
     {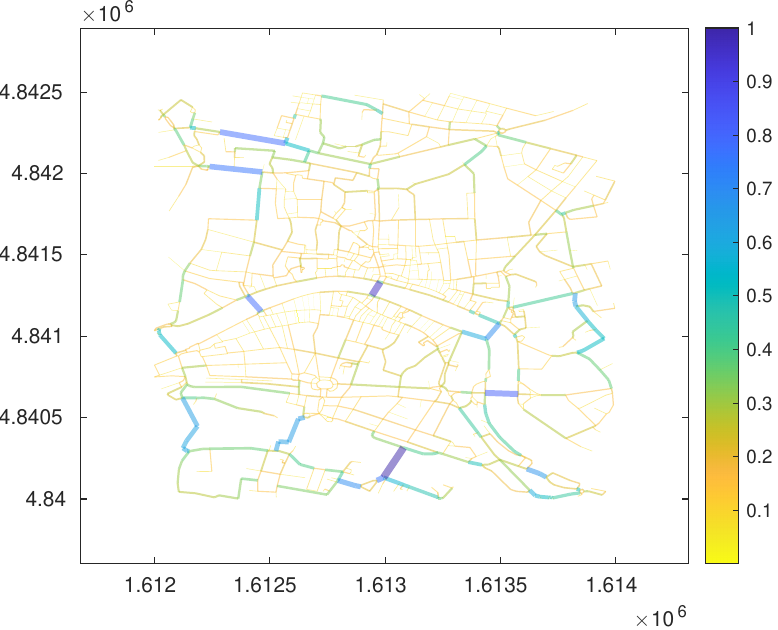}\caption{\small Centralities of the roads in the Pisa road map. On the left, the values are computed by the measure of \cite{abcmp}; on the right, the values are computed using the derivative. The thicker the line, the higher the centrality of the road. 
     }
    \label{fig:pisa}
\end{figure}

\begin{figure}
    \centering
    \includegraphics[width=0.45\linewidth]
    {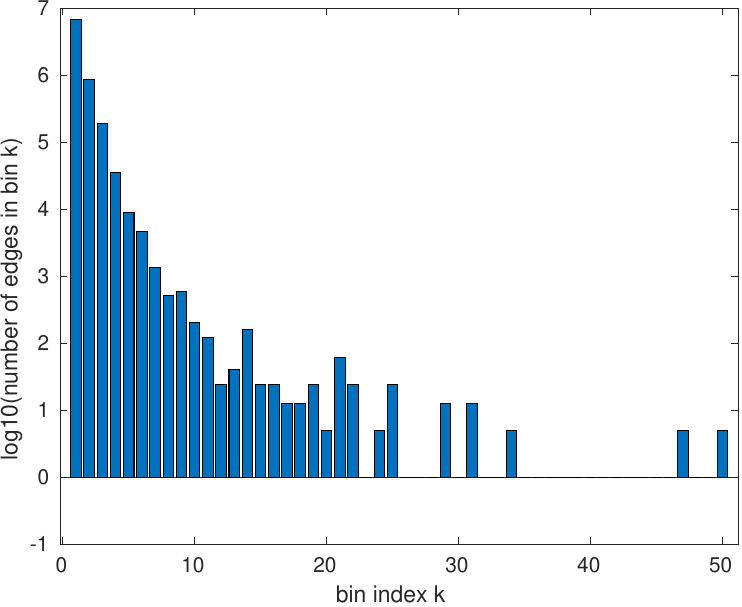}~~
     \includegraphics[width=0.45\linewidth]
     {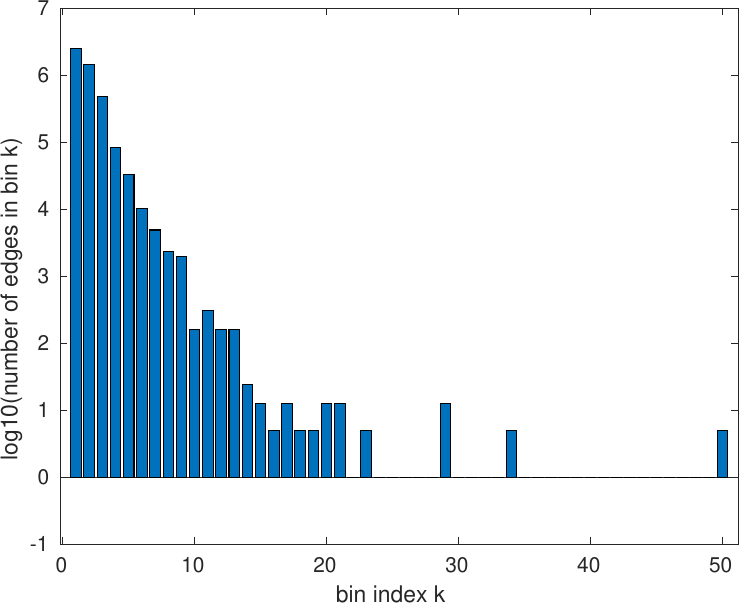}\caption{\small Log-scale histograms of the centralities in the road map of Pisa, divided into 50 bins. On the left, the values are computed by the measure of \cite{abcmp}; on the right, the values are computed using the derivative. 
     }
    \label{fig:histo}
\end{figure}

A second test concerns the computation of the centrality measure $\bar\mu(p,q)$ of the edges and the ``non-edges''. 
In Figure~\ref{fig:nonedge1}, the cases of some specific graphs are considered. Namely,  star graphs, one-path graphs, and cyclic graphs are analyzed with size $n=10$, together with a binary tree graph with 7 nodes. On the left, the graph picture is displayed, on the right, the values of the centralities of edges and non-edges are shown, accompanied by a background color that highlights the centrality level of the corresponding edge.

\begin{figure}
    \centering
    \includegraphics[width=0.45\linewidth]{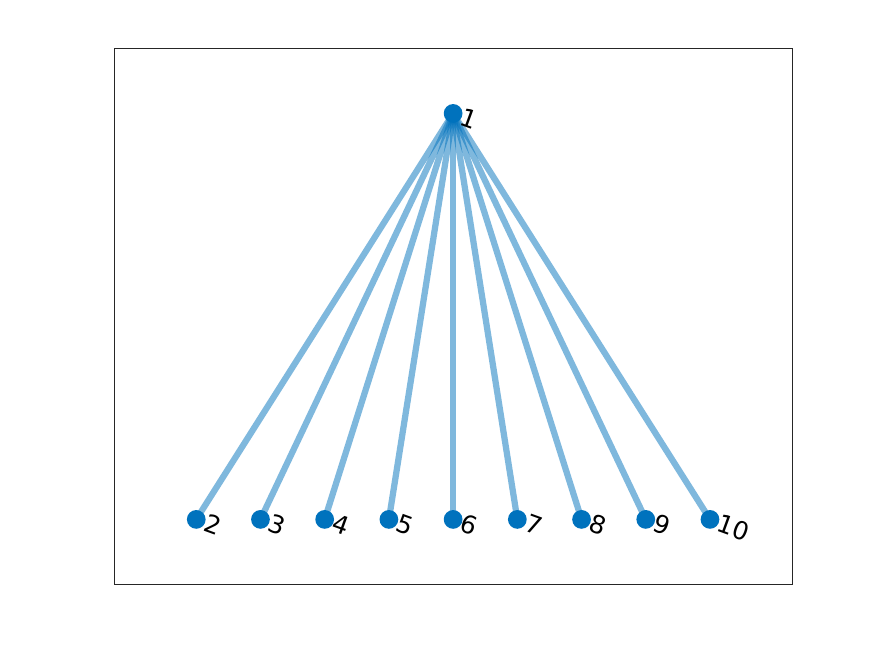}
    \includegraphics[width=0.45\linewidth]{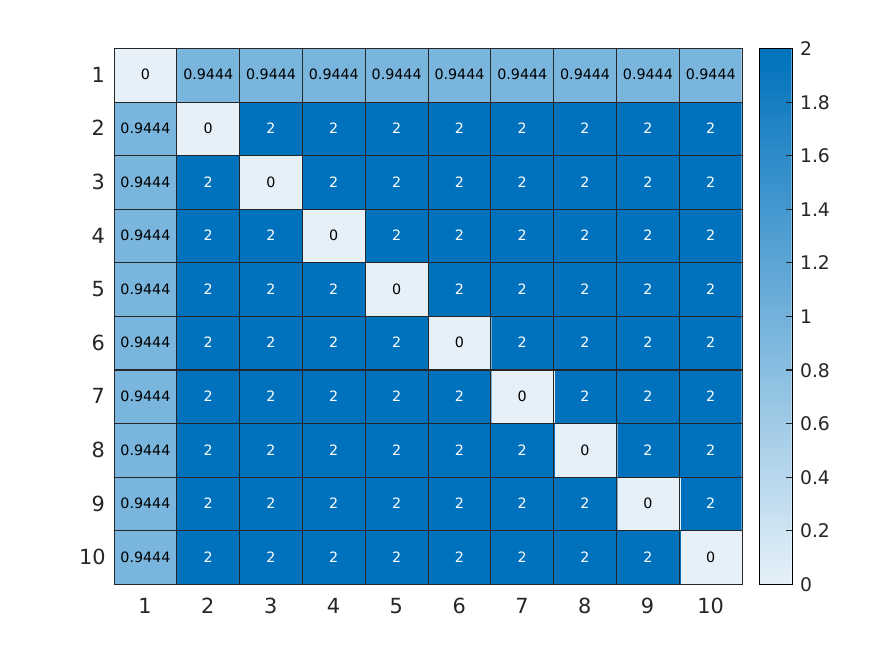}\\
    \includegraphics[width=0.45\linewidth]{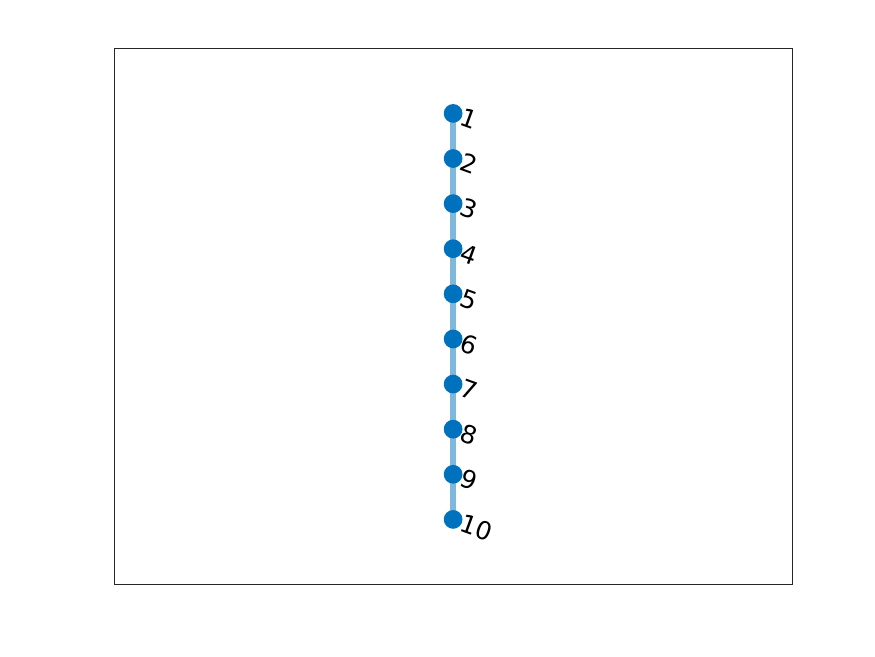}
    \includegraphics[width=0.45\linewidth]{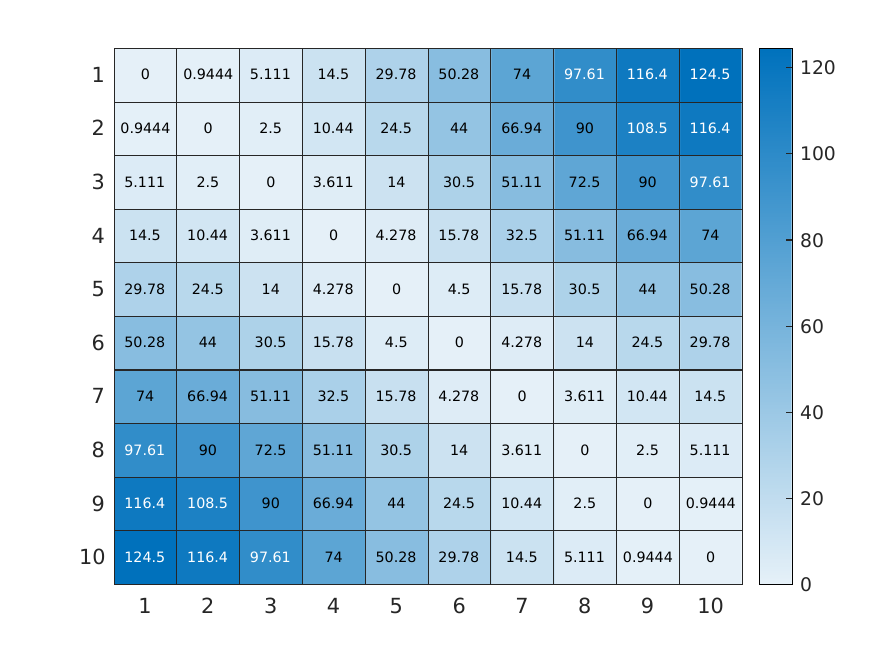}\\
    \includegraphics[width=0.45\linewidth]{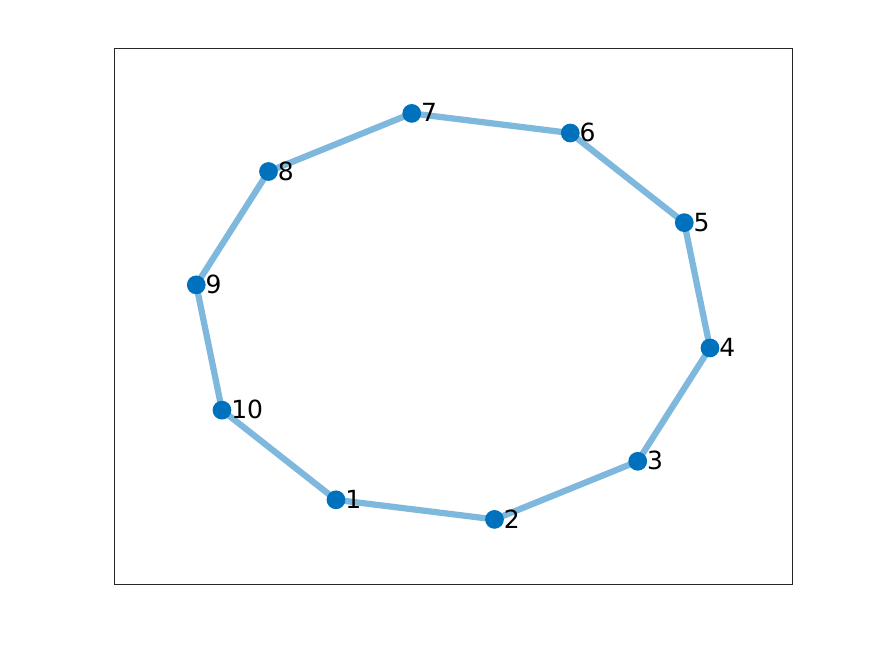}
    \includegraphics[width=0.45\linewidth]{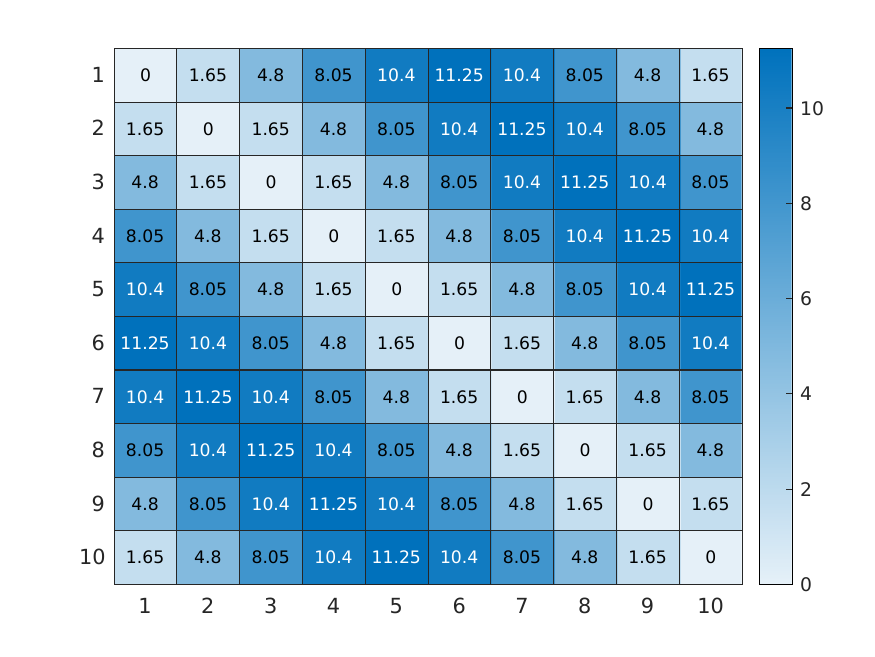}\\
    \includegraphics[width=0.45\linewidth]{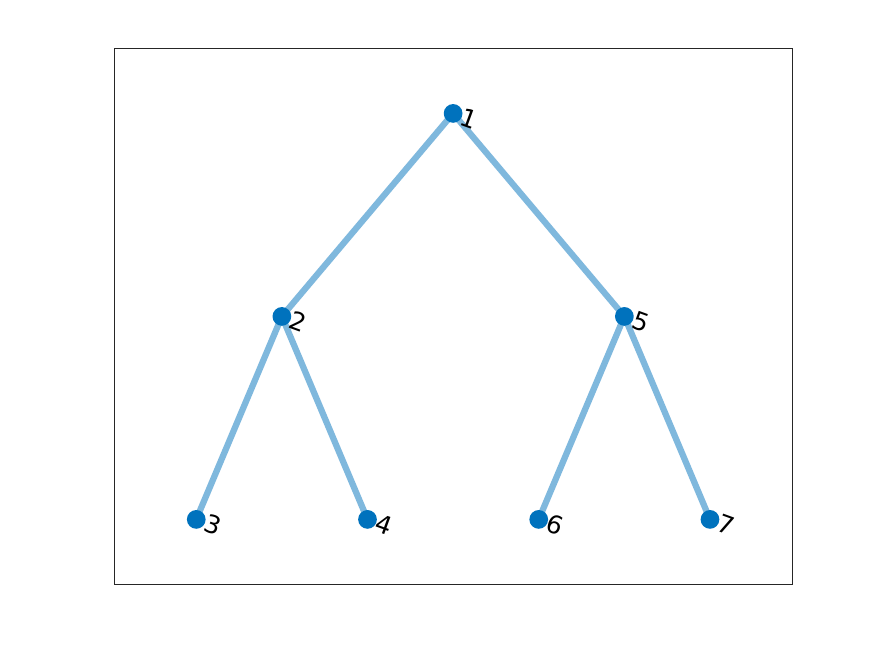}
    \includegraphics[width=0.45\linewidth]{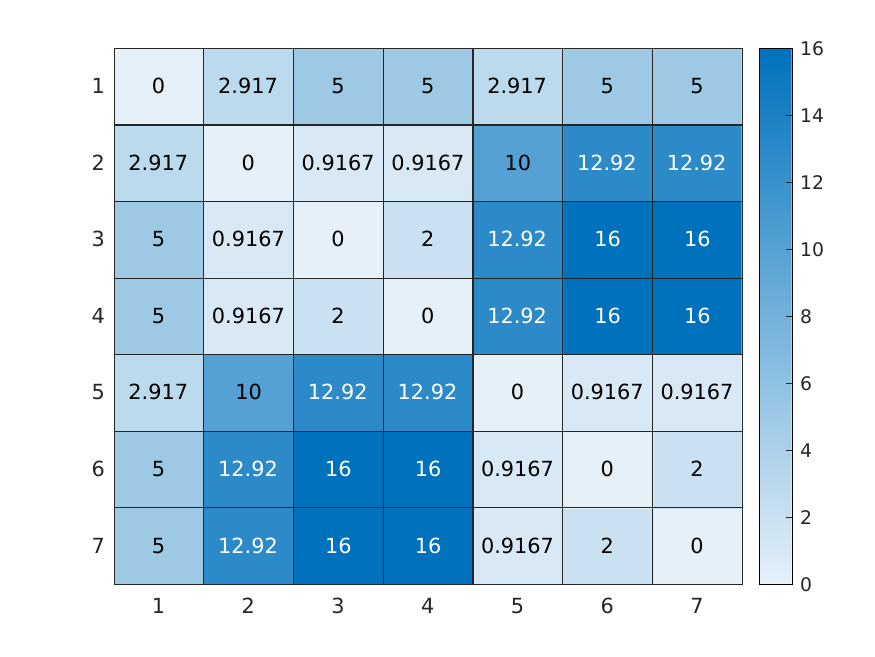}\\

     \caption{\small Edges centralities in star, one-path, circulant, and tree graphs. On the left, the tested graph, on the right, the centralities, in log scale, of edges and non-edges.}
    \label{fig:nonedge1}
\end{figure}

In the case of the star graph, all the non-edges have necessarily the same centrality value, which is slightly higher than twice the value of the existing edges.
For the one-path graph, the most important non-edge is the one that connects the first node to the last one, and the importance of non-edges is roughly proportional to the number of nodes that they allow to 'skip' over.
Similarly, for the cyclic graph, the most important non-edges are the ones that connect a node to its direct opposite. Also, notice that the matrix with the centralities of edges and non-edges is circulant, due to the cyclicity of the problem.
For the binary tree, the most important non-edges are those that connect the leaves having a large relative distance.

Evaluating the centrality of the non-edges is also related to the problem of link prediction in social networks \cite{lp}, where the goal is to predict the likelihood of a future association between two nodes,
knowing that there is no association between the nodes in the current state of the graph. For this problem, the directional derivative of Kemeny's constant with respect to the non-edge $\{p,q\}$ can be interpreted in the following way: the higher the variation, the more unlikely is the possibility that the social relation between node $p$ and node $q$ will be activated.

We compared the measure $\bar{\mu}$ to several existing ones on a link prediction experiment with real-world data: the network of collaboration between network science authors studied in~\cite{Newman2006}\footnote{\url{https://networks.skewed.de/net/netscience}}. This network is a weighted, undirected graph. We restrict our analysis to its largest connected component, with 379 nodes and 914 edges.

Most existing measures used for link prediction, such as the Jaccard and Adamic-Adar indices~\cite{jaccard_adamic} and the resource allocation index~\cite{resource_allocation}, are based only on common neighbours: the probability of a new link between $p$ and $q$ is higher when they have many common neighbours. As a result, the score of any two nodes with no common neighbours is zero, and these measures cannot predict links between them. The so-called \emph{common neighbour centrality}~\cite{common_neighbor} is also based on common neighbours, but with a further additive term that is inversely proportional to the distance between the two nodes; hence, this centrality measure gives a nonzero score also to nodes that have no common neighbours. We compare these four indices to our new measure $\bar\mu$, using their implementations in the Python library \texttt{NetworkX}~\cite{NetworkX}. 
\begin{figure}
    \centering
    \includegraphics[width=0.45\linewidth]{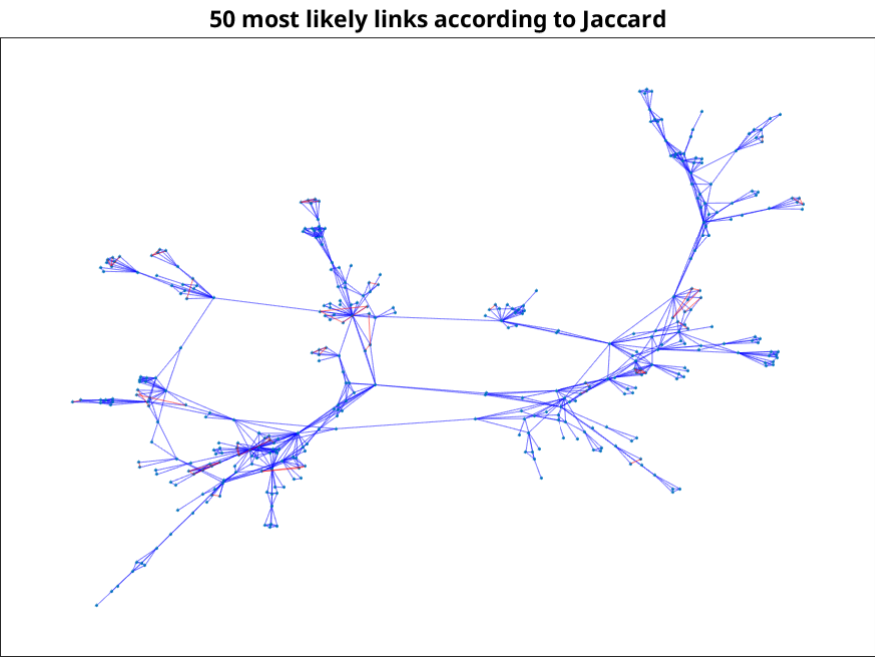}
    \includegraphics[width=0.45\linewidth]{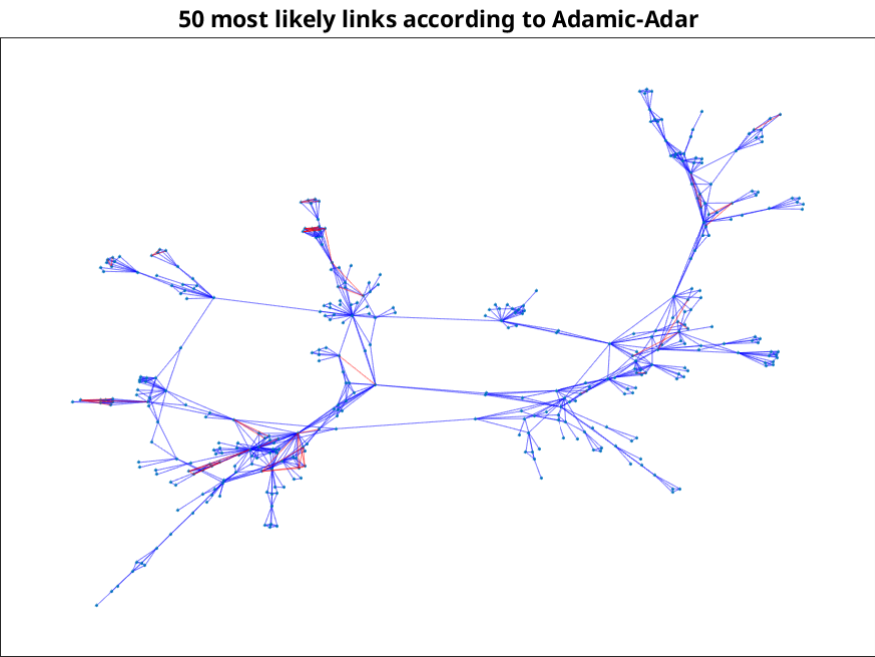}
    \includegraphics[width=0.45\linewidth]{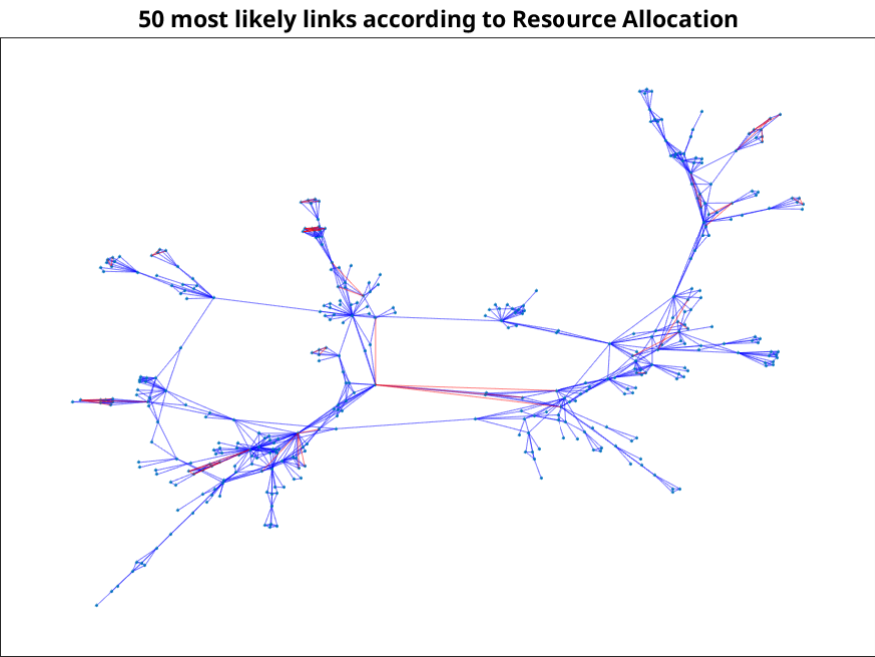}
    \includegraphics[width=0.45\linewidth]{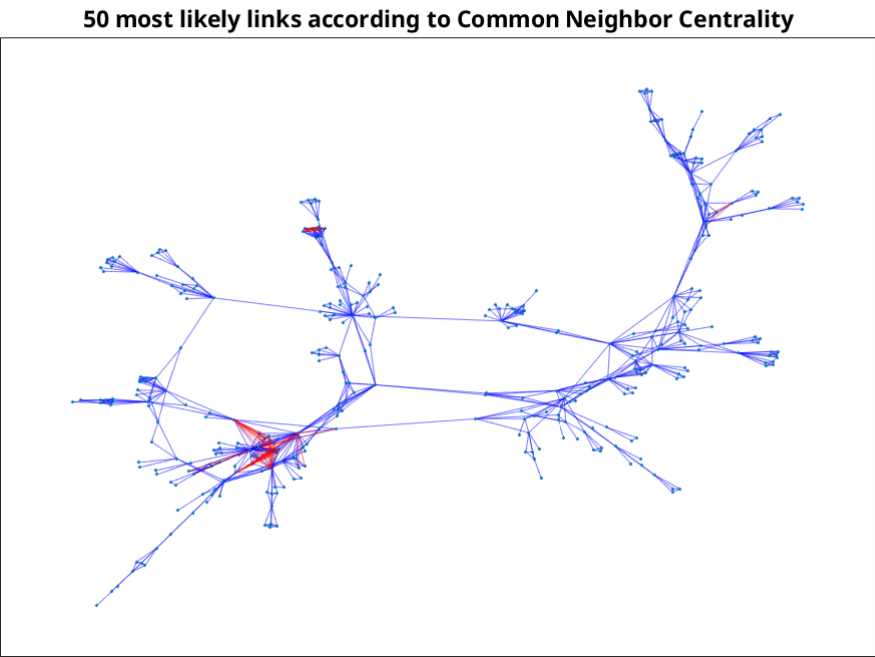}
    \includegraphics[width=0.45\linewidth]{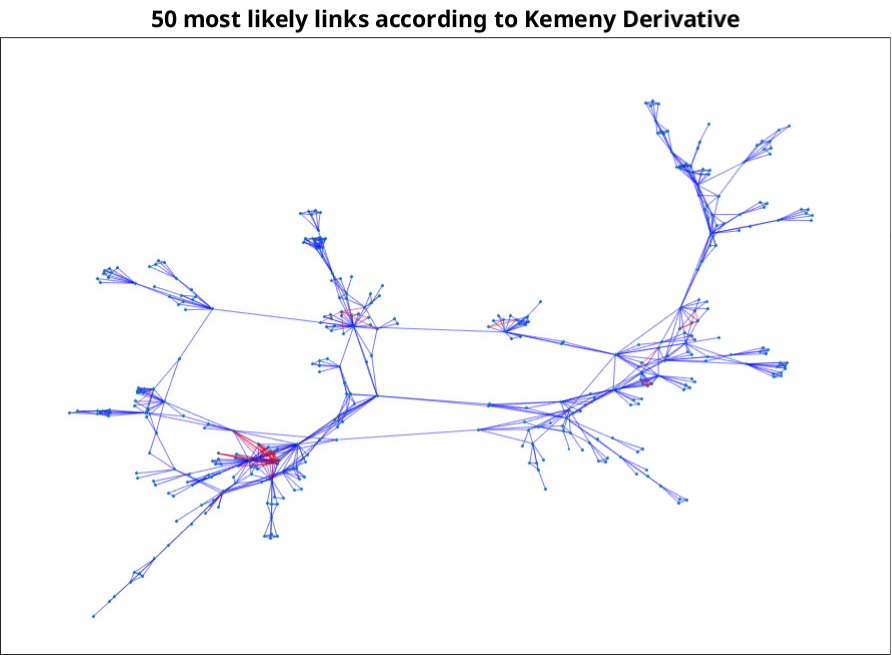}
    \caption{\small 50 top new edges (in red) most likely to be added to the blue graph, as predicted by the five tested measures.}
    \label{fig:newlinks}
\end{figure}

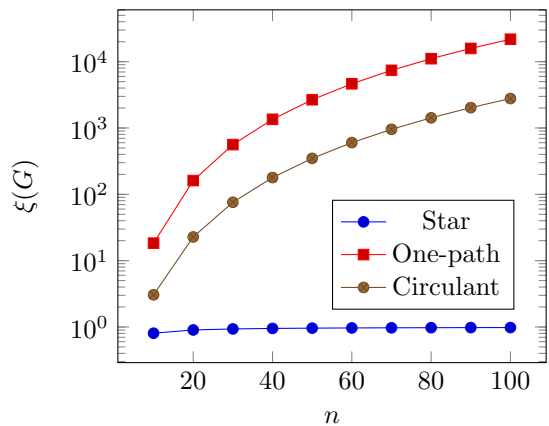
\begin{figure}
\begin{center}
        \begin{tikzpicture}
        \begin{semilogyaxis}[
                legend style={at={(0.5,0.15)},anchor=south west},
                width = .6\linewidth, 
                xlabel = {$n$}, 
                ylabel = {$\xi(G)$}, title = {Global sensitivity of Kemeny's constant}]
            \addplot table[x index = 0, y index = 1] {sens.txt};
            \addplot table[x index = 0, y index = 2] {sens.txt};
            \addplot table[x index = 0, y index = 3] {sens.txt};
            \legend{Star,
              One-path,
               Circulant}
        \end{semilogyaxis}
        \end{tikzpicture}
        \end{center}
\caption{\small Global sensitivity $\zeta(G)$, as a function of the number of vertices $n$, for the star, one-path and circulant graph.}\label{fig:sens}
        \end{figure}
We display in Figure~\ref{fig:newlinks} the new predicted edges on a plot of the original graph. The correlation coefficients between the five measures are shown in Table \ref{tab:1}.

\begin{table}
\centering
\begin{tabular}{lccccc}
\toprule
   & J & A & R & C&  K\\
  \midrule
Jaccard  & 1 & 0.83 & 0.76 & 0.56& -0.15 \\
Adamic--Adar  & 0.83 & 1 & 0.96 & 0.64& -0.17 \\
Resource allocation  & 0.76 & 0.96& 1 & 0.56 & -0.15\\
Common neighbour & 0.56 & 0.64 & 0.56 & 1& -0.59 \\
Kemeny derivative & -0.15 & -0.17 & -0.15 & -0.59& 1  \\
\bottomrule
\end{tabular}\caption{\small The correlation coefficients between pairs of the five measures.}\label{tab:1}
\end{table}

\begin{table}\centering
\begin{tabular}{lcc}
\toprule
1& \color{blue}    RAVASZ, E & \color{blue} YOOK, S\\
2&NEDA, Z & YOOK, S\\
3&GOMEZGARDENES, J & PACHECO, A\\
4&SCHUBERT, A & YOOK, S\\
5&FARKAS, I & YOOK, S\\
6& \color{blue} FERRERICANCHO, R & \color{blue}  MONTOYA, J\\
7& \color{blue} VALVERDE, S & \color{blue}  MONTOYA, J\\
8&RAVASZ, E & TOMBOR, B\\
9&VESPIGNANI, A & WEIGT, M\\
10&BARRAT, A & BOGUNA, M\\
\bottomrule
\end{tabular}\caption{\small  The node labels for the top ten new links predicted by the new measure.}\label{tab:2}
\end{table}

The correlations between the Kemeny derivative and the other four measures are negative, as expected, because they are sorted in opposite ways: a low Kemeny derivative means a high probability of a new link. The three measures based on common neighbours are highly correlated with each other, and only poorly correlated with the new measure. The new measure is moderately correlated with the common neighbour centrality. The node labels for the top ten new links predicted by the new measure $\bar{\mu}$ are shown in Table \ref{tab:2}.
We can verify on literature databases that pairs \#1, \#6, \#7 have a common publication that was not included in the original dataset; we consider this as a confirmation that our measure is effective at link prediction.

For the circulant, star, and one-path graphs represented in Figure~\ref{fig:nonedge1}, we have computed the global sensitivity $\zeta(G)$ defined in \eqref{eq:global}, as a function of the number of vertices $n$. From Figure~\ref{fig:sens}, we observe that $\zeta(G)$ is almost constant for the star graph, while it increases very quickly for the circulant and one-path graph. This behavior is expected from the topology of the three kinds of graphs.

\section{Conclusions}\label{sec:conc}
Edge centrality measures are important for assessing the relevance of roads in a road map. While the measure introduced in \cite{abcmp} is highly effective, its definition and computation must be adjusted for cut-edges \cite{abcmp,bklm}. If not calculated correctly, this measure is susceptible to numerical instability.

We developed a novel, always-positive measure for graph centrality based on the directional derivative of Kemeny's constant. This measure is applicable to all edges, including cut-edges, and avoids the computational issues of cancellation. We have derived an explicit formula for this measure, which relates it to existing measures from \cite{abcmp} and \cite{bklm} through the inverse of a modified Laplacian matrix.

Our measure can also be computed for non-edges, making it useful for link prediction. We have provided explicit expressions for one-path graphs, showing our measure's relationship to edge weights. Numerical experiments confirm its effectiveness as a tool for assessing road map centrality and as a strong alternative for link prediction analysis.

\section*{Funding}
This work has been partially supported 
by:
MUR Excellence Department Project awarded to the Department of Mathematics, University of Pisa, CUP I57G220007\\ 00001;
European Union - NextGenerationEU under the National Recovery and Resilience Plan (PNRR) - Mission 4 Education and research - Component 2 From research to business - Investment 1.1 Notice Prin 2022 - DD N. 104  2/2/2022, titled ``Low-rank Structures and Numerical Methods in Matrix and Tensor Computations and their Application'', proposal code 20227PCCKZ – CUP I53D23002280006. The second and the third author are member of the INdAM GNCS group.

\end{document}